\documentclass[12pt,leqno,twoside]{amsart}
\usepackage{amssymb}
\usepackage{latexsym}
\usepackage{verbatim}
\usepackage{epsfig}
\usepackage[all]{xy}
\usepackage{srcltx}
\usepackage[normalem]{ulem}

\usepackage[colorlinks=true, pdfstartview=FitV, linkcolor=blue, citecolor=blue, urlcolor=blue]{hyperref}
\usepackage{amstext,amsmath,amscd, bezier,indentfirst,amsthm,amsgen,enumerate, geometry}
\usepackage{amsfonts,color}

\newtheorem{theorem}{Theorem}[section]
\newtheorem{corollary}[theorem]{Corollary}
\newtheorem{lemma}[theorem]{Lemma}
\newtheorem{proposition}[theorem]{Proposition}
\theoremstyle{definition}
\newtheorem{definition}[theorem]{Definition}
\theoremstyle{remark}
\newtheorem{remark}[theorem]{\sc Remark}
\theoremstyle{remark}

\theoremstyle{remark}

\theoremstyle{remark}

\theoremstyle{remark}

\theoremstyle{remark}

\renewcommand{\Box}{\square}    %\diamond
\renewcommand{\Bbb}{\mathbb}
\newcommand{\cal}{\mathcal}

\renewcommand{\int}{{\rm{int}}}

\newcommand{\Sing}{{\rm{Sing}}}

\renewcommand{\ker}{\mathop{{\rm{ker}}}\nolimits}

\newcommand{\cl}{{\rm{closure}}}
\newcommand{\grad}{\mathop{\rm{grad}}\nolimits}

\newcommand{\ity}{{\infty}}

\renewcommand{\d}{{\rm{d}}}

\newcommand{\fin}{\hspace*{\fill}$\Box$}
\newcommand{\m}{\setminus}

\newcommand{\cC}{{\cal C}}

\newcommand{\cW}{{\cal W}}

\newcommand{\bC}{{\Bbb C}}

\newcommand{\bP}{{\Bbb P}}
\newcommand{\bN}{{\Bbb N}}

\newcommand{\bX}{{\Bbb X}}
\newcommand{\bY}{{\Bbb Y}}

\newcommand{\C}{{\Bbb C}}

\newcommand{\bx}{\mathbf{x}}

\newcommand{\bt}{\boldsymbol{t}}

\begin{document}

\title[Detecting asymptotic non-regular
values]{Detecting asymptotic non-regular
values by polar curves}

\author{Zbigniew Jelonek}
\address{Instytut Matematyczny,
Polska Akademia Nauk,
\'Sniadeckich 8, 00-956 Warszawa\\
Poland.} \email{najelone@cyf-kr.edu.pl}

\author{Mihai Tib\u ar}
\address{Math\' ematiques, UMR 8524 CNRS,
Universit\'e de Lille 1, \  59655 Villeneuve d'Ascq, France.}
\email{tibar@math.univ-lille1.fr}

\date{\today}
\keywords{polar curve, bifurcation locus, Malgrange non-regular values, effectivity}
\subjclass[2010]{32S05, 32S50, 14D06 14Q20, 58K05, 14P10}
\thanks{The authors acknowledge the support of the Labex CEMPI
(ANR-11-LABX-0007-01) at Lille, and of the CIRM at Luminy through the RIP program. The first author was partially
supported by the NCN grant 2013/09/B/ST1/04162,
 2014-2017.}

\begin{abstract}
 We locate the Malgrange non-regular values of a given polynomial function $f:\bC^n \to \bC$ by using a series of affine polar curves. We moreover show that all
non-trivial Malgrange non-regular values of $f$ are indicated by a single
 ``super-polar curve'' which we introduce here, providing also an effective algorithm of detection.

\end{abstract}

\maketitle
\setcounter{section}{0}

\section{Introduction}\label{s:intro}
Let $f:\bC^n \to \bC $ be a polynomial of degree $d\ge 2$.  Ren\' e
Thom proved that $f$ is a $C^\infty-$fibration outside a finite
set, where the smallest such set is called {\it the bifurcation set of}
$f$ and is  denoted by $B(f)$. Roughly
speaking the set $B_\infty (f)$ consists of points at which $f$ is
not a locally trivial fibration at infinity (i.e., outside a large
ball). Two
fundamental questions appear in a natural way: how to characterize the set $B(f)$ and how to
estimate the number of points of this set.

Let us recall that the set $B(f)$ contains
the set $f(\Sing f)$ of critical values of $f$ and the
set $B_\infty (f)$ of bifurcations points at infinity.

In case $n=2$ there are well-known criteria to detect $B(f)$, see
e.g. \cite{Ti-reg}, \cite{Du}, and there are also
estimations of the upper bound of the number $\# B(f)$ in terms of
the degree or other data \cite{Ha}, \cite{Ha1},  \cite{jel6},
 \cite{LO}, \cite{gw},  \cite{JT} etc.

Whenever $n>2$ one has no exact criteria but one defines
regularity conditions at infinity that each yield some finite set
of values containing $B(f)$ and thus approaching the problem of
estimating $\# B(f)$. To control the set $B_\infty (f)$ one can
use the set of {\it asymptotic critical values of} $f$:
\[K_\infty (f) :=\{ y \in \bC \mid \exists \ (x_l)_{l\in \bN},
\| x_l\| \rightarrow\infty, \ \mbox{such\ that} \ f(x_l)\rightarrow y \ \mbox{and} \ \Vert
x_l\Vert \Vert d f(x_l)\Vert\rightarrow 0\}.\]
 If $c\notin
K_\infty (f)$, then it is  usual to say that $f$ satisfies {\it
Malgrange's condition}  at $c$ ( or $c$ is Malgrange regular). The
set
 $K_\infty(f)$ naturally decomposes into two pieces: the set $TK_\infty(f)$ of trivial
 Malgrange non-regular values which come from the critical
 points of $f$ (i.e.  there  is  a  sequence
$x_l\rightarrow\infty$  such that $x_l\in \Sing(f)$ and $\
f(x_l)\rightarrow y$) and the remaining set $NK_\infty(f) := K_\infty (f)\m TK_\infty(f)$ of
non-trival Malgrange  non-regular values. Of course
$TK_\infty(f)\subset f(\Sing f)$. Since the set $f(\Sing f)$ is relatively
easy to compute, the problem which remains
 is how to compute the set $NK_\infty(f)$.

It was proved (cf \cite{par1}, \cite{JK-crelle}, \cite{jel4}) that one has the inclusion
$B_\infty (f)\subset K_\infty (f)$. Setting $K(f):=f(\Sing f)\cup K_\infty
(f)$,  we get the inclusion $B(f)\subset K(f)$. Estimations of the number of the
Malgrange non-regular values have been given in \cite{JK-crelle}. An algorithmic method for recovering
the set $K_\infty(f)$ has been produced more recently \cite{JK}.

We present here two new methods for detecting
$K_\ity(f)$ and for estimating the number $\# K_\ity(f)$, together with an effective algorithm.
 Our first approach,  based on the use of  a series of \emph{polar curves} and their relation to Malgrange non-regularity via the $t$-regularity,  yields an exact characterisation of the set $NK_\ity(f)$.
 We shall recall the notions and the relevant preliminary results in \S \ref{s:reg}, but let us introduce already here our first main statement.

 Let $\{x_1, \ldots , x_n\}$ be a {\em generic system of coordinates} of $\bC^n$, after Definition \ref{d:polar}. Let us consider the successive restrictions of $f$ to the affine hyperplanes:
\[ f_0 := f, \ \  f_1 := f_{| x_1=0}, \   \ldots , \  f_{n-2} := f_{| x_1=\cdots = x_{n-2}=0}, \]
 and the corresponding generic polar curves $\Gamma (x_i,f_{i-1})$, for $i=1, \ldots , n-1$.

For a  mapping $g: X\to Y$, let $J_g$ denote the \emph{non-properness set} \cite{jel}, \cite{jel1}
(also called the \emph{Jelonek set}) of the mapping $g$, see   \S \ref{s:noether}, Theorem  \ref{setSF}. If
$A\subset X$, then by $J_g(A)$ we denote the non-properness set of
the restriction $g_{|A}$.

We say that an irreducible algebraic variety
$S\subset \bC^n$ is {\it horizontal} if  $f(S)$ is not a
point (i.e. $S$ is not included in some fibre of $f$). The union of all horizontal components of the polar curve
$\Gamma (x_i,f_{i-1})$ will be called the {\it horizontal part} and will be denoted by $H\Gamma (x_i,f_{i-1})$.

We prove the following characterisation of the set $NK_\ity(f)$:

\begin{theorem}\label{t:1}
The set $NK_\infty(f)$ of non-trivial Malgrange non-regular
regular values of $f$ is included in the union of the
non-properness sets of the mapping $f$ restricted to a horizontal
part of the polar curves $\Gamma (x_i,f_{i-1})$, more precisely we
have the equality:
\begin{equation}\label{eq:union}
 NK_\ity(f) = \bigcup_{i=1}^{n-1} J_f(H\Gamma (x_i,f_{i-1})) \setminus J_f(\Sing f).
\end{equation}
\end{theorem}

Note that $J_f(\Sing f)$ equals the the set of critical values of $f$ which are images of fibers containing nonisolated singularities.

\medskip

Let $\Sing f = S_0 \cup S_1\cup \cdots \cup S_r$ be the decomposition of the singular locus into irreducible components, where $S_0$ denotes the union of all point-components (i.e. $S_0$ is the set of isolated singularities of $f$).
 For $i>1$ we denote by $d_i = \deg S_i$  the degree of the positive dimensional component $S_i$.

\begin{corollary}\label{c:bound}
For $d>2$ we have:
 \begin{equation}\label{eq:bound2}
  \# NK_\infty(f) \le \frac{(d -1)^n - 1}{d-2} - \sum_{i=1}^r
 d_i \dim S_i,
 \end{equation}
and for   $d=2$:
 \[ \# NK_\infty(f) \le n-1 - \sum_{i=1}^r d_i \dim S_i.\]
 \fin
\end{corollary}

 According to their definition, the polar curves of the above statement are affine curves,  some of them are maybe empty, and they
 do not detect, in general, values from $TK_\infty(f)$.
The first polar curve $\Gamma (x_1,f)$ detects some Malgrange non-regular value $c\in
NK_\infty(f)$ whenever
 the fiber $f^{-1}(c)$ has only isolated singularities
 at infinity in the sense of Definition \ref{d:t-sing}, cf Theorem \ref{t:polar}.
 However, the first polar curve may not detect all values from
 $NK_\infty(f)$ and that is why we need more polar curves. We explain this phenomenon by the existence of \emph{non-isolated $t$-singularities at infinity}, cf \S \ref{s:proof}. For example, if $f(x,y,x)=x+x^2y$ then the
 polar curve of $f$ is empty, but $f$ has a non-trivial Malgrange
 non-regular value $0$. This example also shows that Theorem 3.6
 in \cite{sa} is not correct.
More precisely, if we use polar curves, then the problem of detecting  non-trivial  Malgrange non-regular values cannot be done in a single step (as was wrongly claimed
in \cite{sa}), but turns out to fall into $n - 1$ steps as we describe now in our Theorem \ref{t:1}, each step being
the detection of the non-properness set of a certain generic polar curve.
 %However, note that the union \eqref{eq:union} is not disjoint.

However, the question ``\emph{is it possible to recover all
non-trivial Malgrange non-regular values in just one single step}'' subsists as a chalenging problem. We solve it positively in the second part of our paper by
 introducing a new and different device called ``super-polar curve''. Let us give here an outline of its construction.
We consider the following polynomials:

$$g_i(a,b)=\sum_j a_{ij}\frac{\partial f}{\partial x_j}+\sum_{j,k} b_{ijk} x_k \frac{\partial f}{\partial x_j}, \ i=1,\ldots,n-1,$$
where $a_{ij}, b_{ijk}$ are complex constants. Let:
\begin{equation}\label{eq:superpolar}
\Gamma_f(a,b):=\cl\{V(g_1,\ldots,g_{n-1})\setminus \Sing(f)\},
\end{equation}
where we use here the Zariski closure. It turns out that, for general  $a_{ij},
b_{ijk}$ the set $\Gamma_f(a,b)$ is  a non-empty curve, which we shall call
\emph{super-polar curve of $f$}. We say that a  component
$S\subset \Gamma_f(a,b)$ is {\it horizontal} if $f(S)$ is not  a single
point. The union of all horizontal components of
$\Gamma_f(a,b)$ will be called the {\it horizontal part of
$\Gamma_f(a,b)$} and will be  denoted by $H\Gamma_f(a,b)$. We obviously have the inclusion  $J_f({H\Gamma_f(a,b)}) \subset  J_f({\Gamma_f(a,b)})$. We prove the
following result:

\begin{theorem}\label{t:2}
The set  $NK_\infty(f)$ of nontrivial Malgrange  non-regular
values of $f$ is included in  the non-properness set of a mapping
$f$ restricted to the horizontal part of a sufficiently general
super-polar curve $\Gamma_f(a,b)$, namely one has the following
inclusion:
\begin{equation}\label{eq:union2}
 NK_\ity(f) \subset  J_f({H\Gamma_f(a,b)}).
\end{equation}

\end{theorem}

 \smallskip

\begin{corollary}\label{c2}
  If $n>2$ and $NK_\infty(f)\not=\emptyset$ then:

 \[\#NK_\infty(f)\le d^{n-1}-1 - \sum_{i=1}^r d_i.\]

In particular
 if $NK_\infty(f)\not=\emptyset$, then:
\[\#K_\infty(f)\le d^{n-1}-1 - \sum_{i=1}^r (d_i-1).\]

In case $n=2$,  if $NK_\infty(f)\not=\emptyset$, then:
 \[\#NK_\infty(f)\le d-2 - \sum_{i=1}^r d_i, \ \  \mbox{ and } \ \ \#K_\infty(f)\le d-2 - \sum_{i=1}^r
 (d_i-1).\]
\end{corollary}

The plan of the paper goes as follows: in \S \ref{s:reg} and \S \ref{s:noether} we develop some
preliminary results in order to prepare the proofs of Theorem
\ref{t:1} in \S \ref{s:proof}, and of Theorem \ref{t:2} in \S
\ref{s-polar}, together with their corollaries, respectively. In
\S \ref{alg} we sketch the algorithm to detect the set
$NK_\infty(f)$ effectively.

%%%%%%%%%%%%%%%%%%%%%%%%

%%%%%%%%%%%%%%%%%%%%%%%%%%%%%%%%%%%%%%%%
\section{Regularity conditions at infinity}\label{s:reg}
% We recall some facts following \cite{Ti-top}, \cite{Ti-reg}.

 \subsection{Malgrange regularity condition at a point at infinity}\label{d:Mcond}
 %%%%%%%%%%%%%%%%%%%%%%%%%%%%%%%%%%%%%%%%

 Pham formulated in \cite[2.1]{Ph} a regularity condition
which he attributed to  Malgrange.  We recall the localized version at
infinity, after \cite{Ti-top}, \cite{Ti-reg}.

 We identify $\bC^{n}$ to the graph of
$f$, namely $X :=  \{(x,\tau) \in \bC^{n}\times \bC \mid f(x) =\tau\}$, and consider its algebraic closure in $\bP^{n}\times \bC$, which is the hypersurface:
\begin{equation}\label{eq:X}
  \bX = \{\tilde
f(x_0, x) - \tau x_0^{d}= 0  \} \subset \bP^{n}\times \bC,
\end{equation}
where $x_0$ denotes the variable at infinity,  $d = \deg f$ and $\tilde f(x_0, x)$ denotes the homogenization of degree $d$ of $f$ by the variable $x_0$.
Let $\bt: \bX \to \bC$ denote the restriction to $\bX$ of the second projection $\bP^n\times\bC\rightarrow\bC$,  a proper
 extension of the map $f$.  We denote by $\bX^\ity = \bX\setminus X$ the
 divisor at infinity defined in each affine chart by the equation $x_0 =0$.

%%%%%%%%%%
\begin{definition}\cite{Ti-reg}
 Let  $\{\bx_i\}_{i\in \bN} \subset
\bC^n$ be a  sequence of points with the following properties:
\begin{itemize}
\item[(L$_1$)]  \hspace{5mm}
$\| \bx_i\| \to \infty$ and $f(\bx_i) \to \tau$, as $i\to \infty$.
\item[(L$_2$)]  \hspace{5mm}
$ \bx_i \to p \in \bX^{\infty}$, as $i\to \infty$.
\end{itemize}
One says that the fibre $f^{-1}(\tau)$  verifies the {\em Malgrange condition} if
there is $\delta > 0$ such that, for any sequence of points with
property (L$_1$) one has
\begin{itemize}
\item[(M)]  \hspace{5mm}
$ \|  \bx_i\|\cdot \| \grad f(\bx_i)\| > \delta$.
\end{itemize}
We say that $f$ verifies {\em Malgrange condition at} $p\in
\bX^{\infty}$ if there is $\delta_p >0$ such that one has (M) for any sequence of
points with property (L$_2$).
\end{definition}

%%%%%%%%%%%%%%% 2.8 %%%%%%

\begin{remark}\label{n:interpret}
It  follows from the definition that $f^{-1}(\tau)$  verifies the Malgrange
condition if and only if $f$ verifies  Malgrange condition (M) at any point $p = (z,\tau) \in
\bX^{\infty}\cap \bt^{-1}(\tau)$ and for the same positive constant $\delta_p = \delta$.
\end{remark}

%%%%%%%%%%%%%%%%%%%%%%%%

%%%%%%%%%%

\subsection{Characteristic covectors and $t$-regularity}\label{ss:char_covect}

We recall the notion of $t$-regularity from \cite{Ti-cras}, \cite{Ti-reg}. Let $H^{\infty}=\{[x_0:x_1:\ldots:x_n]\in\bP^n\mid x_0=0\}$ denote the hyperplane at infinity
and let $\bX^{\infty}:=\bX\cap (H^{\infty}\times\bC)$.

We consider the  affine charts $U_j\times\bC$ of $\bP^n\times\bC$, where $U_j=\{x_j\neq0\}$, $j=0,1,\ldots,n$. Identifying the chart $U_0$ with the affine space $\bC^n$, we have  $\bX\cap (U_0\times\bC)=\bX\setminus\bX^\infty= X$ and $\bX^\infty$ is covered by the charts $U_1\times\bC,\ldots, U_n\times \bC$.

If $g$ denotes the projection to the variable $x_0$ in some affine chart $U_j\times\bC$, then the
 \emph{relative conormal} $C_{g}(\bX\backslash \bX^{\infty} \cap U_j\times\bC) \subset \bX \times \check\bP^{n}$ is well defined (see e.g. \cite{Ti-asymp}, \cite{Ti-book}), with the projection $\pi(y,H) = y$, where $\check\bP^{n}$ is identified to the projective space of hyperplanes in $U_j\times\bC$.
Let us then consider the space $\pi^{-1}(\bX^\ity)$ which is well defined for every chart $U_j\times\bC$ as a subset of $C_{g}(\bX\backslash \bX^{\infty} \cap U_j\times\bC)$. By \cite[Lemma 3.3]{Ti-top}, the definitions coincide at the intersections of the charts.

\begin{definition}\label{d:car_covectors_space}
We call  {\em space of characteristic covectors at infinity} the well-defined set $\mathcal{C}^{\ity} :=\pi^{-1}(\bX^\ity)$. For some $p_0\in \bX^\ity$, we denote $\cC^\ity_{p_0} := \pi^{-1}(p_0)$.
\end{definition}

Considering now the second projection $\bt:\bP^n\times\bC\rightarrow\bC$ in place of the function $g$ in the above consideration, we obtain the relative conormal space $C_{\bt}(\bP^n\times\bC)$. Then we have:

\begin{definition}\label{d:t-reg} \cite{Ti-top}
 We say that $f$ is \emph{$t$-regular} at $p_0\in\bX^\infty$ if $C_{\bt}(\bP^n\times\bC)\cap \cC^\ity_{p_0}=\emptyset$ or, equivalently, $\d \bt \not\in \cC^\ity_{p_0}$.
\end{definition}

\begin{definition}\label{d:t-sing}
We say that $f$ has \emph{isolated $t$-singularities at infinity} at the fibre $f^{-1}(t_0)$ if
this fibre has isolated singularities in $\bC^n$ and if the set
\[ \Sing^\ity f := \{ p\in \bX^\ity \mid f^{-1}(t_0)
\mbox{ is not } t\mbox{-regular at $p$}\}\]
 is a finite set.
\end{definition}

It follows from the definition that $\Sing^\ity f$ is a closed algebraic subset of $\bX^\ity$,  see e.g.
 \cite{Ti-top}, \cite{Ti-book}, \cite[\S 6.1]{DRT}.  By the algebraic Sard Theorem, the image $\bt(\Sing^\ity f)$ consists of a finite number of points.

We need the following key equivalence in the localized setting (proved initially in \cite[Proposition 5.5]{ST} and \cite[Theorem 1.3]{Pa-m}, as explained in \cite{Ti-reg}):
\begin{theorem}\cite[Prop. 1.3.2]{Ti-book}\label{t:malgrange-t-reg}
 A polynomial $f: \bC^n \to \bC$ is $t$-regular at $p_0\in\bX^\infty$ if and only if
$f$ verifies the Malgrange condition at this point.
\fin
\end{theorem}

 More precisely we have the following relations, cf \cite{Ti-asymp}, \cite{Ti-book}:
\begin{equation} \label{eq:reg}
 \mbox{Malgrange regularity} \Longleftrightarrow t\mbox{-regularity}
\Longrightarrow
\rho_E\mbox{-regularity} \Longrightarrow \mbox{topological triviality}
\end{equation}
which also extend to polynomial maps $\bC^n \to \bC^p$ as shown in \cite{DRT}.

%%%%%%%%

%%%%%%%%%%%%%%%%%%%%%%%%%
\subsection{Polar curves and $t$-regularity}\label{s:polar}

We define the affine polar curves of $f$ and show how they are
related to the $t$-regularity condition, after \cite{Ti-asymp}.

 Given a polynomial $f: \bC^n \to \bC$ and a
linear function $l: \bC^n \to \bC$, the \emph{polar curve of $f$ with respect to $l$}, denoted by  $\Gamma (l,f)$, is the closure in $\bC^n$
of the set $\Sing (l,f)\setminus \Sing f$, where $\Sing (l,f)$ is the critical locus of the
map $(l,f): \bC^n\to \bC^2$.
Denoting by $l_H :\bC^n \to \bC$ the
unique linear form (up to multiplication by a constant) which defines a hyperplane $H\in \bC^n$ (also regarded as a point in the   projective space $\check\bP^{n-1}$ of linear hyperplanes in $\bC^n$), we have the following genericity result of Bertini-type.

\begin{lemma}\label{l:polar} {\rm \cite[Lemma 1.4]{Ti-top}}\\
There exists a Zariski-open set $\Omega_{f,a} \subset \check\bP^{n-1}$ such that, for any $H\in \Omega_{f,a} $ and some fixed  $a\in \bC$,  the polar set $\Gamma(l_H,f)$ is a curve or it is an empty set, and no component is contained in the fibre $f^{-1}(a)$.
\fin
\end{lemma}

%%%%%%%%%%

\begin{definition} \label{d:polar}
For $H\in \Omega_{f,a} $, we call  $\Gamma (l_H,f)$ the {\em generic affine polar curve} of $f$
with respect to $l_H$. A system of coordinates $(x_1, \ldots , x_n)$ in $\bC^n$ is
called {\em generic} with respect to $f$ iff $\{ x_i=0\} \in \Omega_{f,a} $, $\forall i$.
\end{definition}
 It follows from  Lemma \ref{l:polar} that such systems of coordinates are generic among all linear systems of coordinates.

%The isolated $t$-singularities are characterised by the following
%%%%%%%%%%%%%%%%%%
%%%%%%%%% 3.5 %%%%%%%%%%%%%%%%
 Let $\overline{\Gamma (l_H,f)}$ and $\overline{\Sing f}$  denote the algebraic closure in $\bX$ of the respective sets. We then have:

\begin{theorem}\label{t:polar}
Let  $f:\bC^n\to \bC$ be a polynomial function and let $p\in \bX^\ity$, $a:=\bt(p)$.
\begin{enumerate}
\rm \item \it If $p$ is a $t$-regular point then $p\not \in \overline{(\Gamma (l_H,f)} \cup \overline{\Sing f})\cap \bX^\ity$, for any $H\in \Omega_{f,a} $.
\rm \item \it Let  $p$ be either  $t$-regular or an isolated $t$-singularity at infinity.
Then $p$ is a $t$-singularity at infinity if and only if  $p \in \overline{\Gamma (l_H,f)}$ for some $H\in \Omega_{f,a} $.
\end{enumerate}
\end{theorem}

\begin{proof}
The result and its proof can be actually extracted from \cite[\S 2.1]{Ti-book}. More precisely:

\noindent (a) follows from \cite[Prop. 2.1.3]{Ti-book} and \cite[(2.1), pag.17]{Ti-book}.

\noindent
(b) follows by combining Thm. 2.1.7, Thm. 2.1.6 and Prop. 2.1.3 from \cite[\S 2.1]{Ti-book}.

\end{proof}

 The above theorem means that isolated $t$-singularities at infinity are precisely detected by the horizontal part of the generic polar curve.
In case $p\in \bX^\ity$ is a non-isolated $t$-singularity (which occurs whenever $n>2$), the general affine
polar curve $\Gamma(l_H,f)$ might not contain the chosen point $p$ in its closure at infinity. We shall show in the next section how to deal with this situation.
%But in this case we may slice. Indeed, the $t$-singular set is an algebraic set, contained in a finite number of fibres of $\bt$,  and may be endowed with a Whitney stratification
%%%%%%%%%%%%%%%%%%%%%%%%%%%%%%%%%%%%
%%%%%%%%%%%%%%%%%%

\section{Proof of Theorem \ref{t:1}}\label{s:proof}

%\eqref{eq:bound2}
%%%%%%%%%%%%%%%%%%%%%%%%%%%%%%%%%%%%%%%%

Let $B := \Sing^\ity f \cap (\bX \setminus \cup_{a\in f(\Sing f)}\bX_a)$.
By Theorem \ref{t:malgrange-t-reg}, we have the equality: $$NK_\ity(f) =  \bt(B).$$

By Theorem \ref{t:polar}(a) and Theorem \ref{t:malgrange-t-reg}, if the generic polar curve $\Gamma(l_H,f)$ is nonempty, then it  intersects the hypersurface $\bX^\ity$ at finitely many points
and these points are $t$-singularities, hence Malgrange non-regular points at infinity.
  %We proceed by induction upon $\dim \Sing^\ity f$.

Let us first assume that $\dim B =0$.
Then, by Theorem  \ref{t:polar}, for  $p\in B$ (which by our assumption is an isolated $t$-singularity),  the generic polar curve
 passes through $p$, so this point is ``detected'' by the horizontal part of the polar curve $\overline{\Gamma (x_1,f)}$, for some generic choice of the coordinate $x_1$ (in the sense of Definition \ref{d:polar} and Lemma \ref{l:polar}). Therefore, in the notations of the Introduction,
 the corresponding asymptotic non-regular value belongs to $J_f(H\Gamma (x_1,f))$.

Therefore,  in our case $\dim B =0$, the equality \eqref{eq:union} follows from Theorem \ref{t:polar} and Theorem \ref{t:malgrange-t-reg}.

Let us now treat the case $\dim B >0$. We will show \eqref{eq:union} by a double inclusion.

\medskip
\noindent \textbf{The inclusion ``$\supset$''.}   Let us first prove the inclusion $J_f(H\Gamma (x_i,f_{i-1})) \setminus J_f(\Sing f)  \subset NK_\ity(f)$ for each $i=1, \ldots, n-1$.
We proceed by a ``reductio ad absurdum" argument. Assume that $a \not \in NK_\ity(f)$ and denote
$\bX^\ity_a := \bX^\ity \cap \bt^{-1}(a)$.

(a). If $a\in J_f(H\Gamma (x_1,f_{0})) \setminus J_f(\Sing f)$, then there exist points $p\in \bX_a^\ity \cap \overline{\Gamma (x_1,f_{0})}$. By Theorem \ref{t:polar}(a), this means that $p$ is a $t$-non-regular point, which implies in turn that  $a \in NK_\ity(f)$, by Theorem \ref{t:malgrange-t-reg}.

(b). Assume that $a\not\in J_f(H\Gamma (x_i,f_{i-1})) \setminus J_f(\Sing f)$ for $i=1, \ldots, k-1$ (for some $k\ge 2$), and that
$a\in J_f(H\Gamma (x_k,f_{k-1})) \setminus J_f(\Sing f)$.

 We endow the hypersurface $\bX\subset \bP^n \times \bC$ with a finite complex Whitney stratification $\cW$ such that $\bX^\ity := \{ f_d = 0\} \times \bC$ is a  union of strata. Our Whitney stratification at infinity is also Thom (a$_{x_0}$)-regular, by \cite[Theorem 2.9]{Ti-top}, where $x_0 =0$ is some local equation for $H^\ity$ at $p$.

There exists a Zariski-open set  $\Omega'\subset \check
\bP^{n-1}$ of linear forms $\bC^n \to \bC$ such that, if $H\in \Omega'$, then $(H^\ity \cap \overline{H}) \times \bC$ is transversal in $H^\ity \times \bC$ to all
strata of $\cW$ in the neighbourhood of $\bX_a^\ity$.
Due to the Thom (a$_{x_0}$)-regularity of the stratification, it follows that slicing by $H\in \Omega'$  insures  the $t$-regularity of the restriction $f_{|H}$  at any point $p\in (\overline{H}\times \bC) \cap \bX_a^\ity$.  More precisely, from our hypothesis  $\d \bt \not\in \cC^\ity_{p}$  (see Definition \ref{d:t-reg}) we deduce that $\d \bt' \not\in \cC'^\ity_{p}$ for $p \in \overline{H} \cap \bX_a^\ity$, where $H\in \Omega'$, $\bt' :=  \bt_{|\overline{H}\times \bC}$ and $\cC'^\ity$ is the space of  Definition \ref{d:car_covectors_space} starting with the restriction $f_{|H}$ instead of $f$.
This implies that $a \not \in NK_\ity(f_{|H})$.

By taking $H\in \Omega' \cap \Omega_{f, a}$ we get in addition that $a \not\in J_{f}(\Sing f_{|H})$.  We denote $f_1 := f_{|H}$.

Now, if $k=2$ in our first assumption at point (b),  we may apply the reasoning (a) to $f_1$ in place of $f$ and obtain $a \in NK_\ity(f_1)$, hence a contradiction.

In case $k>2$,  after applying the slicing process (b) exactly $k-2$ more times, namely successively to $f_1, \ldots , f_{k-2}$, we arrive to the similar contradiction for $f_{k-1}$.

\medskip
\noindent \textbf{The inclusion ``$\subset$''.}
 Let  $a\in \bt(B)$ be an asymptotic non-regular value such that the set of $t$-singularities in $\bX^\ity_a$ is not isolated. More precisely, according to Definition \ref{d:t-sing},
this set is equal to $\bX^\ity_a \cap \Sing^\ity f$. From the remark after Definition \ref{d:t-sing}, it follows  that $\bX^\ity_a \cap \Sing^\ity f$ is an algebraic set.
Let therefore  $k := \dim \bX^\ity_a \cap \Sing^\ity f$ be its dimension, where $k>0$ by our assumption $\dim B>0$.  We show how to reduce $k$ one by one until zero.

For that we use two facts:

\noindent
(a). From the above proof of the first inclusion we extract the fact that if $\d \bt \not\in \cC^\ity_{p}$ then $\d \bt' \not\in \cC'^\ity_{p}$,  for any $H\in \Omega'$, where $\bt' :=  \bt_{|\overline{H}\times \bC}$.\\
\noindent
(b). Moreover, by a Bertini type argument\footnote{based on the fact that the relative conormal $T^*_{\bt_{|\cW_i}}$ is of dimension $n-1$, the same as $\check \bP^{n-1}$.}, there exists a Zariski-open set $\Omega''\subset \check \bP^{n-1}$ such that if $H\in\Omega''$ then  $H\times \bC$ is transversal to any stratum $\cW_i\subset \bX^\ity$ of the Whitney stratification except at finitely many points.

For some $H\in \Omega'\cap \Omega''$ we consider the restriction $f_{|H}$ and  the space similar to $\bX$ defined at \eqref{eq:X} attached to the polynomial function $f_{|H}$,  which we denote by $\bY$. These two facts imply the equality:
\[   \dim (\bY^\ity_a \cap \Sing^\ity f_{| H}) = \dim (\bX^\ity_a \cap \Sing^\ity f) - 1,\]
as long as $k>0$ (which is our assumption).  This shows the reduction to $k-1$.

%Note that $\dim \bY_a \cap\Sing f_{|H} \le 0$, i.e., the affine singularities of the fibre $f_{|H}^{-1}(a)$ are at most isolated,  %again by the genericity of the slicing.
 We thus continue to slice by generic hyperplanes and lower one by one the dimension of the set $\Sing^\ity f$ until we reach zero, thus we slice a total number of $k$ times. The restriction of $f$ to these iterated slices identifies to
the restriction $f_{k}$ defined in the Introduction.

After this iterated slicing we have $f_k$ with a nonempty set of isolated $t$-singularities at infinity over $a$, each of which are detected by the horizontal part of the polar curve $\overline{\Gamma (x_{k+1},f_{k})}$,  like shown in the first part of the above proof. We therefore get $a\in J_f(H\Gamma (x_{k+1},f_{k})$.

Altogether this shows the inclusion:
 $ NK_\ity(f) \subset \bigcup_{i=1}^{n-1} J_f(H\Gamma (x_i,f_{i-1})) \setminus J_f(\Sing f).$
Our proof of Theorem \ref{t:1} is now complete.
\fin
%%%%%%%%%%%%%%%%%%%%%%%%
%%%%%%%%%%%%%%%%%%%%%%%%

\subsection{Proof of Corollary \ref{c:bound}}

We estimate the number of Malgrange non-regular values
$K_\infty(f)$ given by Theorem \ref{t:1}. Let us fix a generic
system of coordinates $(x_1, \ldots ,x_n)$.  The following equations:
\begin{equation}
 \frac{\partial f_d}{\partial x_2} = 0, \ldots, \frac{\partial f_d}{\partial x_n} = 0
\end{equation}
define the algebraic set $\Gamma (x_1,f) \cup \Sing f \subset \bC^n$ of degree $(d -1)^{n-1}$.
Therefore, if nonempty,  $\Gamma (x_1,f)$ is a curve of degree $\le (d -1)^{n-1}$. After Bezout, the curve $\overline{\Gamma (x_1,f)}$
will meet a non-degenerate hyperplane, and in particular the
hyperplane at infinity, at a number of points which is bounded from above
by   $(d -1)^{n-1} - \sum_{i=1}^{r}d_{i}$. Repeating this procedure after successively slicing by  general hyperplanes like explained in the above proof, we finally add up the numbers of solutions. This gives the following sum:
 \begin{equation}\label{eq:sum}
  (d -1)^{n-1}  + (d -1)^{n-2} + \cdots + (d -1) = \frac{(d -1)^n - 1}{d-2}
 \end{equation}
to which we have to substract the sums of degrees of the positive dimensional irreducible components of $\Sing f$ and their successive slices. It follows that we substract the degree $d_i$ a number of $\dim S_i$ times which corresponds to the number of times we slice $S_i$ and drop its dimension one-by-one until we reach dimension 0.
This proves Corollary \ref{c:bound}. \fin

%\sout{
%%%%%%%%%%%%%%%%%%%%%%%%%%%%%%%%%%%%%%%%%%%%%%%%%%%%%%%%%%%%%%%%%%%%%
\medskip

\subsection{New bound for the number of atypical values at infinity}

In \cite[Corollary 1.1]{JK} one finds the following upper bound
for Malgrange non-regular values:
\begin{equation}\label{eq:bound1}
\# K_\infty(f) \le \frac{d^n - 1}{d+1}.
\end{equation}

Our estimation \eqref{eq:bound2} yields to the following one for $K_\infty(f)$:
\begin{equation}\label{eq:boundK}
  \# K_\infty(f) \le \frac{(d -1)^n - 1}{d-2} - \sum_{i=1}^r
 d_i \dim S_i \ + r .
 \end{equation}
  This is somewhat sharper than
\eqref{eq:bound1}. Both have the highest degree term $d^{n-1}$  and the coefficient of the term $d^{n-2}$ in our formula is
smaller for high values of $n$.

%%%%%%%%%%%%%%%%%%%%%%%%%%%%%%%%%%%%%%%%%%%%%%%%%%%%%%%%
%%%%%%%%%%%%%%%%%%%%%%
\section{The non-properness set and the generalized Noether lemma}\label{s:noether}

In this section we give the preliminary material which will lead to the definition in \S \ref{s-polar} of the ``super-polar curve''.

    If $f: X \to Y$ is a dominant, generically finite polynomial  map of smooth affine varieties, we  denote  by $\mu (f)$ the number of points in a generic fiber of $f$. If  $\{ x \}$  is  an  isolated
component  of  the  fiber $f^{-1}  (f(x))$, then  we  denote by ${\rm mult}_x (f)$ the
multiplicity of  $f$ at  $x$.

Let $X, Y$ be  affine varieties, recall that a  mapping
$f:X\rightarrow Y$ is {\it not proper} at  a point $y \in Y$ if
there  is  no  neighborhood $U$  of  $y$  such  that $f^{-1}
(\overline{U}))$ is compact. In other words, $f$ is  not proper at
$y$ if there is  a sequence $x_l\rightarrow\infty $ such that
$f(x_l)\rightarrow y $. Let $J_f$ denote the set of points at
which the mapping $f$ is  not proper.  The   set $J_f$ has the
following properties (see \cite{jel}, \cite{jel1}, \cite{jel3}):

\begin{theorem} \label{setSF}
 Let $X\subset \Bbb C^k$ be an irreducible variety of dimension $n$ and let
$f = (f_1,\ldots, f_m): X \rightarrow \Bbb C^m$ be a
generically-finite polynomial  mapping. Then  the  set $J_f$ is an
algebraic subset of $\Bbb C^m$ and it is either empty or it has
pure dimension $n-1$. Moreover, if $n=m$ then
 $${\rm deg} \ J_f
\leq \frac{ \deg X(\prod^n_{i=1} {\rm deg} \ f_i) - \mu (f)}
{\min_{1\le i \le n}\deg\ f_i}.$$

\end{theorem}

In the case of a polynomial map of normal affine varieties it  is
easy to  show  the following:

\begin{proposition}\label{q-f}
Let $f: X \rightarrow Y$ be a dominant   and quasi-finite
polynomial map of normal affine varieties. Let $Z\subset Y$ be an
irreducible subvariety which is not contained in $J_f$. Then every
component of the set $f^{-1}(Z)$ has dimension dim $Z$,  and if $g$ denotes the restriction of $f$ to $f^{-1}(Z)$,  then
$$J_g=J_f\cap Z.$$
\fin
\end{proposition}

\begin{proof}
By the Zariski Main Theorem in version of Grothendieck, there is
an affine variety $\overline{X}$, which contains $X$ as a dense
subset and a regular finite mapping $F:\overline{X}\to Y$ such
that $F_{|X}=f.$ Since  the   mapping $F$ is finite, all
components of $F^{-1}(Z)$ have dimension dim $Z.$ Now the
condition $Z\not\subset J_f$ implies that  all components of
$f^{-1}(Z)$ have dimension dim $Z.$ Let $S:=\overline{X}\setminus
X.$ Observe that $J_f=F(S).$ Moreover, $J_g=F(S\cap
F^{-1}(Z))=F(S)\cap Z.$
\end{proof}

Let $M^n_m$ denotes the set of all linear forms $L:
\bC^m\to \bC^n$. We need the following result,
which is a modification of   \cite[Lemma 4.1]{jel5}:

\begin{proposition}\label{noether}\rm (Generalized Noether Lemma) \it \\
Let $X\subset \Bbb C^m$ be an affine variety of dimension $n$. Let
$A\subset \Bbb C^m$ be a line and $B\subset X$ be a subvariety
such that $A\not\subset B$. Let $x_1 : \bC^m \to \bC$ be a linear
projection and assume  that $x_1$  is non-constant on $X$ and on
$A$.  Let $a_1,\ldots,a_s\in A\cap X$ be some fixed set of points.

 There exist a Zariski open dense subset $U\subset
 M_m^{n-1}$ such that for every $(n-1)$-tuple $(L_1,\ldots,L_{n-1})\in U$ the
mapping  $\Pi = (x_1,L_1,\dots,L_{n-1}): X\to \Bbb C^{n}$ satisfies
the following conditions:
\begin{enumerate}
\rm \item \it the fibers of $\Pi$ have dimension at most one,

\rm  \item \it there is a  polynomial $\rho\in \Bbb C[t_1]$ such that
$$J_\Pi=\{ (t_1,\dots,t_{n})\in \Bbb C^{n} \mid \rho (t_1)=0\},$$

\rm  \item \it $\Pi(A)\not\subset \Pi(B),$

\rm  \item \it all fibers $\Pi^{-1}(\Pi(a_i)), \ i=1,\ldots, s$ are finite and
non-empty.
\end{enumerate}
\end{proposition}

\begin{proof}
For any $Z \subset \bC^m$, denote by $\tilde Z$ the projective
closure of $Z$ in $\bP^m$, and let  $H^\infty$ denote  the hyperplane at infinity.
Then
$\dim  \tilde{X} \cap  H^\infty= n-1$.

Hence there is a non-empty Zariski open subset $U_1\subset  M_m^{n-1}$ of $(n-1)$-tuples of linear forms such
that for any $L=(l_1,\ldots,l_{n-1})\in U_1$ we have
      $\dim \tilde{X} \cap H^\infty \cap \ker \ L \le 0$.

Let $l_{n}$ be a general linear form.  Since the $(n+1)$ linear forms $(x_1, l_1, \ldots , l_{n})$ are
algebraically dependent on $X$, there exists a non-zero polynomial
$W\in \Bbb C[T,T_1,\dots,T_{n}]$ such that
we have $W(x_1,l_1,\dots,l_{n})= 0$ on $X$. Let us define:

\begin{equation}\label{eq:L}
 L_i :=l_i-\alpha_i l_{n}, \mbox{ for } i=1, \ldots , n-1 ; \  \alpha_i\in \bC^*.
\end{equation}

 Operating on $W$ the linear change of coordinates $l_i \mapsto L_i$, for sufficiently general coefficients $\alpha_i\in \bC$, we then get a relation:
\begin{equation}\label{*}
 l_{n}^N \rho(x_1) + \sum^{N}_{j=1} l_{n}^{N-j}
A_j(x_1,L_1,\dots,L_{n-1})=0,
\end{equation}
where $N$ is some positive integer, $\rho$ and $A_j$ are polynomials, such that $\rho \not\equiv 0$.

 The map  $P=(x_1,L_1,\dots,L_{n-1},l_{n}) : X\to \Bbb C^{n+1}$ is finite and proper, since
$(L_1,\dots, L_{n-1}, l_{n})$ is so. Let $X':=P(X)$ and consider the
projection:
\[\pi : X' \to \bC^{n}, \ \ (x_1,\dots,x_{n+1})\mapsto (x_1,\dots, x_{n}).\]
 Note that the mapping $\pi$ has fibers of dimension at
most one. From the above constructions it follows that the non-properness locus of the projection $\pi$ is:
 \[J_\pi=\{ (t_1,\dots,t_{n})\in \Bbb C^{n} \mid \rho (t_1)= 0\},\]
for the polynomial $\rho\in \Bbb C[t_1]$ defined by the relation \eqref{*}, since $J_\pi$ is precisely the locus of the values of $(x_1, L_1, \ldots , L_{n-1})$ such that the equation \eqref{*} has less than $N$ solutions for $l_n$, counted with multiplicities.

Let us remark that the genericity conditions on $(l_1,\ldots ,l_{n-1},
l_{n})\in  M_m^{n}$ and the condition that $(\alpha_1,\ldots ,
\alpha_{n-1})$ ensure the non-triviality of the polynomial $\rho$
in (\ref{*}), yield a constructible subset $S$ of $\bC^{n-1}\times
M_m^{n}.$ The algebraic mapping:
\[  \Psi: S \to M_m^{n-1}, \ \ \ (\alpha_1,\ldots, \alpha_{n-1}; l_1,\ldots,l_{n-1}, l_{n})\mapsto (L_1,\ldots ,L_{n-1})\]
where $L_i$ are defined in \eqref{eq:L},
 has a
constructible image  $\Psi(S) \subset M_m^{n-1}$ which contains $U_1$ in its closure, thus  $\Psi(S)$ contains a non-empty Zariski-open subset $U_2$ of  $ M_m^{n-1}$.

We thus obtain (a) and (b) for $U:= U_2$ and for $\Pi :=\pi\circ P$.

\medskip

Next, let us show that there is a non-empty  Zariski open subset
included in $U_2$ such that condition (c) is also satisfied.

 Note that dim $B\le n-1$. Moreover, there is a point $a \in A \setminus B$, such that the dimension of  $B_a:= B\cap
x_1^{-1}(x_1(a))$ is $<n-1$. Let $\Lambda\subset \bC^m$ be the Zariski
closure of the cone over $B_a$ with vertex $a$,  $C_a B_a := \bigcup_{x\in B_a} \overline{ax}$, which is of dimension $\le n-1$. Hence
\[
 \dim \tilde{\Lambda}  \cap H^\infty  < n-1.        %\begin{color}{red} why \le n  ? \end{color}
\]
Consequently, there is a Zariski open subset $U_3\subset U_2$ such
that for $L=(L_1,\ldots,L_{n-1})\in U_3$ we have
      dim $ \tilde{\Lambda} \cap H^\infty \cap \ker L=\emptyset$.
    This means that for $\Pi :=(x_1,L_1,\ldots,L_{n-1})$ we have  $\Pi(a)\not\in \Pi(B)$, which finishes the proof of (c).

Let us finally show that there is an eventually smaller non-empty
Zariski open subset $U\subset U_3$ such that (d) is satisfied too.
Let $D_i :=x_1^{-1}(x_1(a_i))$, for $i= 1, \ldots , s$. Since
$\dim D_i = n-1$, the Zariski closure $D$ of
 $\bigcup_{i=1}^s D_i$ has dimension $n-1$. Hence
\[
\dim H^\infty \cap \tilde{D} < n-1.
\]
Like in the above argument,  there is a Zariski open subset
$U\subset U_3$ such that for $L=(L_1,\ldots,L_{n-1})\in U$ we have
      dim $\tilde{D} \cap H^\infty \cap \ker L =\emptyset$.
      Consequently, for any $i= 1, \ldots , s$, the fiber  $\Pi^{-1}(\Pi(a_i))$ is finite and
       non-empty.
\end{proof}

\begin{definition}\label{d:base}
In the notations of Proposition \ref{noether}, we call \emph{base-set  of non-properness of linear projections} of $X$ with respect to $x_1$, the set:
\[ B(x_1,X):=\bigcap_{L\in U}  J_{(x_1,L)}.\]
\end{definition}

\begin{remark}
If non-empty, the set $B(x_1,X)$ is a finite union
 of hyperplanes of the form $\{b_i\}\times \bC^{n-1}$, by Proposition \ref{noether}(b).
\end{remark}

%%%%%%%%%%%%%%%%%%%%%%%%
\section{Super-polar curve and proof of Theorem \ref{t:2}
}\label{s-polar}
We have defined at \eqref{eq:superpolar} the \emph{super-polar curve $\Gamma_f(a,b)$}
as the Zariski closure  of $V(g_1,\ldots, g_{n-1})\setminus \Sing(f)$, where

\begin{equation}\label{eq:super}
g_i(a,b):=\sum_{j=1}^n a_{ij}\frac{\partial f}{\partial x_j}+\sum_{j,k =1}^n b_{ijk} x_k \frac{\partial f}{\partial x_j}, \ \ i=1,\ldots,n-1.
\end{equation}

That for  general $a_{ij}, b_{ijk}\in \bC$  this is indeed a non-degenerate curve  follows in particular from the next result, which is equivalent to Theorem \ref{t:2}.
Let us recall that
$H\Gamma_f(a,b)$ denotes the horizontal part of $\Gamma_f(a,b)$.

\begin{theorem}\label{t:mainsuperpolar}
There is a Zariski open non-empty set $\Omega$ in the space of
parameters $(a,b)\in \bC^{n(n+1)}$ such that:
\begin{enumerate}
\rm  \item for $(a,b)\in \Omega$ the set $\Gamma_f(a,b)$ is a non-empty curve,
\item $NK_\ity(f) \subset  J_f({H\Gamma_f(a,b)})$.
\end{enumerate}
\end{theorem}

\begin{proof}
Let  $\Phi : \bC^n \to \bC \times
\bC^{n(n+1)}$ be the polynomial mapping defined by:
\[\Phi = \left( f,\frac{\partial f}{\partial x_1},\ldots ,
\frac{\partial f}{\partial x_n}, h_{11} ,h_{12},\ldots ,h_{nn}\right) ,
\]
where $h_{ij}=x_i\frac{\partial f}{\partial x_j},\,
i=1,\ldots,n,\,j=1,\ldots,n$.

Let us observe that $\Phi$ is a
birational mapping (onto its image), in particular it is
generically finite, since $\Phi$ is injective outside the
critical set of $f$.

Let $A:=\bC\times \{
(0,\ldots,0)\}\subset \bC\times \bC^{n(n+1)}$.  %In the following we identify this
%line with a copy of $\bC$ as well as with its image by the
%projection $\Pi$ which we shall introduce below.
By the definitions of  $K_\infty (f)$ and of $\Phi$, we have the equality:
\begin{equation}\label{eq:K}
  K_\infty (f)=A\cap J_\Phi,
\end{equation}
   where $J_\Phi$ denotes the set of
points at which the mapping $\Phi$ is not proper. Recall that $K_\infty (f)$ is finite, hence  the set $A\cap
J_\Phi$ is finite too.

Let $X:=\overline{\Phi(\bC^n)} \subset \bC\times \bC^{n(n+1)}$ and
$B:=J_\Phi$.  Let $B(x_1,X)$ be
a base-set of non-properness of linear projections of $X$ with
respect to $x_1$ (cf Definition \ref{d:base}).

In the following we identify the target $\bC$ of $f$ with the line $A\subset \bC\times \bC^{n(n+1)}$.

Let then $\{p_1,\ldots ,p_s\}:=NK_\infty(f)\cup (B(x_1,X)\cap
A)\setminus f(\Sing f) \subset A\cap X$. By Proposition
\ref{noether} and using its notations, for general $(L_1, \ldots , L_{n-1}) \in U\subset M^{n-1}_{1+ n(n+1)}$,  the mapping:
  $$\Pi = (x_1,L_1,\dots,L_{n-1}): X\to \Bbb C^{n}$$
satisfies the following conditions:
\begin{enumerate}
\item the fibers of $\Pi$ have dimension at most one,
  \item  there is a  polynomial $\rho\in \Bbb C[t_1]$ such that
$$J_\Pi=\{ (t_1,\dots,t_{n})\in \Bbb C^{n} \mid \rho (t_1)=0\},$$
  \item  $\Pi(A)\not\subset \Pi(B),$
 \item all fibers $\Pi^{-1}(\Pi(p_j)), \ j=1,\ldots, s$ are finite and
non-empty.
\end{enumerate}
Let us write $L_i=c_ix_1+l_i(a,b), \ i=1,\ldots,n-1$, where the linear form
$l_i(a,b)$ does not depend on variable $x_1.$ Note that:
\[\Pi(A)=\{
x \in \bC^n \mid  x_1=t, x_2=c_1t,\ldots, x_n=c_{n-1}t, \ \ t\in \bC\}.\]

 For $\Psi :=\Pi\circ\Phi$, we have (see \eqref{eq:K}):
\[\Pi(K_\infty(f))= \Pi(A  \cap J_\Phi) \  \subset \ \Pi(A)\cap
J_\Psi.\]

% Now identify $A$ with a copy of $\bC$ and let
%$c_1,\ldots,c_r\in A$ be all critical points of $f$.

Let $V :=\{y\in\bC^n \mid \dim \Psi^{-1}(y)>0\}$. Since the fibers
$\Psi^{-1}(\Pi(p_j))$, $j=1,\ldots, s$, are finite and non-empty we
have $\Pi(p_j)\not\in \overline{V}$ for $j=1,\ldots,s$. So  let $S$ be
a hypersurface in $\bC^n$ which contains $\overline{V}$ but does
not contain the set of points $\{\Pi(p_1),\ldots, \Pi(p_s)\}$ and
let
\[ R :=S\cup  \{ y \in
\bC^n \mid \prod^r_{c\in \Pi(f(\Sing f))} (y_1-c)=0\}. \]

 With these notations, the mapping
\[\Psi': \bC^n\setminus \Psi^{-1}(R) \to \bC^n\setminus R ,\ \
 x \mapsto \Psi(x)\]
 is quasi-finite, and moreover $\Pi(NK_\infty(f))\subset
J_{\Psi'}$.

Let $\Gamma':=\Psi'^{-1}(\Pi(A))$. By Proposition \ref{q-f},
$\Gamma'$ is a curve and $\Pi(NK_\infty(f))$ is contained in the
non-properness set of the mapping $\Psi|_{\Gamma'} : \Gamma'\to
\Pi(A)\setminus R$.
Consequently, the set $\Pi(NK_\infty(f))$ is also
contained in the non properness set of the mapping $\Psi$ restricted
to $\overline{\Psi^{-1}(\Pi(A)\setminus \Pi(f(\Sing f))}$.

By the definition of $\Psi$ we have
$\Psi^{-1}(\Pi(A))=\Phi^{-1}(\Pi^{-1}(\Pi(A)))$, where:
\[ \Pi^{-1}(\Pi(A))=\{x\in X \mid
l_1(a,b)(x_2,\ldots,x_{n(n+1)})=0,\ldots ,l_{n-1}(a,b)(x_2,\ldots ,x_{n(n+1)})=0\}.\]

Comparing to the definition \eqref{eq:super}, we see that the set
%$\overline{\Psi^{-1}(\Pi(A))\setminus\Sing(f)} =
$\overline{\Phi^{-1}(\Pi^{-1}(\Pi(A)))\setminus \Sing(f)}$
coincides with the super-polar curve $\Gamma_f(a,b)$.

The set $\Gamma_f(a,b)$ is a curve since it is union of the curve
$\overline{\Gamma'}$, which actually coincide with the horizontal
part $H\Gamma_f(a,b)$, and, eventually, some of the one dimensional
fibers of $\Psi$.

Let us now consider a linear isomorphism:
\[T: \bC^n \to \bC^n, \ \   (x_1,\ldots, x_n) \mapsto (x_1, x_2-c_2x_1,\ldots, x_n-c_nx_1).\]

From the above construction we know that  $\Pi(NK_\infty(f))\subset
J_{\Psi_{|\Gamma'}}$.
We then have the inclusion $T(\Pi(NK_\infty(f)))\subset J_{T\circ\Psi_{|\Gamma'}}.$
But $T\circ
\Psi_{|\Gamma'}$ coincides with $f$ on  $\Gamma' = H\Gamma_f(a,b)$,  and
$T(\Pi(NK_\infty(f)))$ coincides with  $NK_\infty(f).$  This shows the inclusion $NK_\ity(f) \subset  J_f({H\Gamma_f(a,b)})$ and ends the proof of point (b) of our theorem.
\end{proof}

\subsection{Proof of Corollary \ref{c2}} \ \\

 We use the terminology of the above proof. We have actually shown
 that if $NK_\infty(f)\not=\emptyset$ then the curve
 $\overline{\Gamma'}$ is non-empty, and that the set $NK_\infty(f)$ is
 contained in the non-properness set of the restriction
  $f_{|\Gamma'}$. The curve $\overline{\Gamma'}$ is
 a subset of the super-polar curve $\Gamma_f(a,b)$ for general coefficients $a$
 and $b$, and moreover, $f$ is constant on all other components of  $\Gamma_f(a,b)$.
 By the generalized Bezout Theorem we have $\deg \Gamma_f(a,b)\le  d^{n-1} - \sum_{i=1}^r d_i$,
thus $\deg \overline{\Gamma'}\le  d^{n-1} - \sum_{i=1}^r
 d_i$. Note that the cardinality of the non-properness set of  $f_{|\Gamma'}$ is estimated by the number of these points at infinity of a curve
 $\Gamma'$ which are transformed by $f$ into $\C.$
 Consequently,  the cardinality of the non-properness set of  $f_{|\Gamma'}$ is bounded from above
 by  the number $d^{n-1}-1 - \sum_{i=1}^r d_i$. We can substract 1 in this formula since actually each branch of $\overline{\Gamma'}$ intersects the hyperplane at infinity also at the value infinity of $f$.
 Thus we also have $\#NK_\infty(f)\le d^{n-1}-1 - \sum_{i=1}^r d_i$.
 Since every connected positive-dimensional component of the critical set
 $\Sing f$ is contained in one fiber of $f$ thus indicates a trivial non-regular value, we
obtain:
\[\#K_\infty(f)\le d^{n-1}-1 - \sum_{i=1}^r (d_i-1).\]
For $n=2$, it turns out that the Malgrange condition can be
recovered (see \cite{Ha}, \cite{Ha1}, \cite{LO}) by the asymptotic
behavior of the derivatives of $f$ only. We thus consider, instead
of the mapping $\Phi$ of the proof of Theorem
\ref{t:mainsuperpolar}, the new mapping
 $\Phi(x,y)=(f(x,y),\frac{\partial f}{\partial x}, \frac{\partial f}{\partial y})$.
  This mapping is generically finite if
 $NK_\infty(f)\not=\emptyset$.  In this case, arguing as above we get the last inequalities of our Corollary \ref{c2}.

%%%%%%%%%%%%%%%%%%%%%%%%
\section{Algorithm}\label{alg}
%%%%%%%%%%%%%%%%
We present here a fast  algorithm which yields a finite set
$S\subset \C$ such that $NK_\infty(f)\subset S$, for a given
polynomial $f: \C^n\to\C.$  By our results, this problem reduces
to computing the non-properness set of the mapping $f_{|\Gamma}:
\Gamma\to \C$ where $\Gamma$ is a super-polar curve of $f.$

Let us first show how to compute the non-properness set $J_g$
of the mapping $g: X\to \C$, where $X\subset \C^n$ is a curve. The
following result can be found in \cite{gr}:

\begin{theorem}\label{th}
If ${\cal B} = (b_1, \ldots , b_t)$ is the  Gr\"obner basis of the
ideal $I\subset k[x_1, \ldots , x_n]$ with the  lexicographic order in
which $ x_1 > x_2 > \ldots
> x_n$, then for every  $0\le m\le n$ the set ${\cal B} \cap
k[x_{m+1},\ldots , x_n]$ is the   Gr\"obner basis of the ideal
$I \cap k[x_{m+1},\ldots , x_n].$  \fin
\end{theorem}

\begin{corollary}
Consider   the ring $\C [x_1,\ldots ,x_n;y_1,\ldots ,y_m]$. Let $V\subset
\C^n\times\C^m$ be an algebraic set and
      let $p: \C^n\times\C^m\rightarrow \C^m$
denote the projection. Assume that ${\cal B}$ is a Gr\"obner basis
of the ideal $I(V)$ with the lexicographic order. Then ${\cal
B}\cap \C [y_1,\ldots ,y_m]$ is a Gr\"obner basis of the ideal
$I(p(V))$.  %=I(\cl(p(V)))$.
\end{corollary}

\begin{proof} Observe that $I(p(V))=I(V)\cap \C [y_1,\ldots , y_m]$
and then  use Theorem \ref{th}.
\end{proof}

Let then $I(X) :=(h_1,\ldots , h_r)$ be the ideal of our curve
$X$. The graph $G\subset \C^n\times \C$ of the non-constant mapping $f: X\to \C$ is given by the
ideal $I=(h_i=0,\ i=1,\ldots , r;\ f(x)-z)\subset \C[x_1,\ldots ,
x_n, z].$

Let $O$ be the order in $\C[x_1,\ldots , x_n, z]$ such that
$x_1>x_2>\ldots  > x_i>x_{i+1}> \ldots  > x_n>z$ . Let ${\cal B}$ denote the
Gr\"obner basis of $I$ with respect to the order $O.$
  Let
$f_i\in {\cal B}\cap \C[x_i,z]$ be a non-zero polynomial which depends on $x_i$. Then:
\[f_i=x_i^{n_i}a_0^i(z)+x_i^{n_i-1}a_1^i(z)+\ldots +a_{n_i}^i(z).\]
By  \cite[Prop. 7]{jel}, \cite[Th. 3.10]{jel1},  for our
mapping $f: X\to \C$ we have:

\[J_f=\bigcup^n_{i=1} \{z\in \C \mid a_0^i(z)=0\}.\]

\medskip

With this preparation, we  now state the algorithm:\\

\noindent
{\em Special case}:  $\Sing(f)$ is a finite set.

\bigskip

\noindent
INPUT: the polynomial $f:\C^n\to\C$

\begin{itemize}
\item[(1)]
choose random coefficients $\alpha_i^k, \alpha_{ij}^k, \
k=1,\ldots , n-1;  \ i,j=1,\ldots , n$.

\item[(2)]  put $g_k=\sum_j \alpha_{j}^k\frac{\partial f}{\partial
x_j}+\sum_{i,j} \alpha_{ij}^k x_i \frac{\partial f}{\partial
x_j}.$

\item[(3)]  put $W:=(g_1,..., g_{n-1})\subset\C[x_1,...,x_n]$, if dim $W>1$
then go back to (1).

\item[(4)] compute a Gr\"obner basis ${\cal B}$ of the ideal
$I=(g_1,\ldots , g_{n-1}, f-z)\subset \C[x_1,\ldots , x_n,z]$ with
respect to order $O$ (as defined above).

\item[(5)]  let $f_i=x_i^{n_i}a_0^i(z)+x_i^{n_i-1}a_1^i(z)+\ldots +a_{n_i}^i(z)\in
{\cal B}_i\cap \C[x_i,z]$  be a non zero polynomial which depends on $x_i$.

\item[(6)] let $S :=\bigcup^n_{i=1} \{z\in \C \mid a_0^i(z)=0\}$.  The set $S$
is the non-properness set of the mapping $f$ restricted to $\{
g_1=0,\ldots , g_{n-1}=0\}$).
\end{itemize}

\noindent
OUTPUT: a finite set $S\subset \C$ such that $NK_\infty(f)\subset
S.$

\bigskip

\noindent In the general case, in order to grip the super-polar curve,  we have to remove from the set
$\{g_1=0,\ldots, g_{n-1}=0\}$ the singular set $\Sing(f).$ To do this, it is
enough to remove the hypersurface $\{ \sum \beta_j \frac{\partial
f}{\partial x_j}=0\}$, where the coefficients $\beta_j$ are
sufficiently general. Indeed such a hypersurface does contain
$\Sing(f)$ but does not contain any component of $\Gamma(a,b).$

\bigskip

\noindent
{\em General  case}:

\bigskip

INPUT: the polynomial $f:\C^n\to\C$

\begin{itemize}
\item[(1)]
 choose random coefficients $\alpha_i^k, \alpha_{ij}^k, \beta_i,
\ k=1,..., n-1,\ i,j=1,\ldots ,n$.

\item[(2)] put $g_k=\sum_j \alpha_{j}^k\frac{\partial f}{\partial
x_j}+\sum_{i,j} \alpha_{ij}^k x_i \frac{\partial f}{\partial
x_j}.$

\item[(3)] put $h=\sum^n_{j=1} \beta_j \frac{\partial f}{\partial x_j}.$

\item[(4)] put  $W:=(g_1,\ldots , g_{n-1}, th-1)\subset \C[t, x_1,\ldots , x_n]$;
if dim $W>1$, then go back to (1).

\item[(5)] compute a Gr\"obner basis ${\cal B}$ of the ideal $I=(th-1,
g_1,\ldots , g_{n-1}, f-z)\subset \C[t,x_1,\ldots , x_n,z]$ with
respect to the order $O$   such that $t>x_1>x_2>\ldots
>\hat x_i>x_{i+1}> \cdots  > x_n>>z$.

\item[(6)] let $f_i=x_i^{n_{i}}a_0^{i}(z)+x_i^{n_i-1}a_1^{i}(z)+\cdots
+a_{n_{l}}^{i}(z)\in {\cal B}\cap \C[x_i,z]$ be a non zero
polynomial which depends on $x_i$.

\item[(7)] let $S=\bigcup^n_{i=1} \{z\in \C \mid a_0^{i}(z)=0\}$. Here $S$
is the non-properness set of the mapping $f$ restricted to $\{
g_1=0,\ldots , g_{n-1}=0\}\setminus \{h=0\}$.
\end{itemize}

OUTPUT: a finite set $S\subset \C$ such that $NK_\infty(f)\subset
S.$

\begin{remark}
The above algorithm is probabilistic (without certification),
hence for really random coefficients $\alpha$ and $\beta$ it gives
a good subset $S(\alpha, \beta)$, but for some choices it can produce a bad answer.
However generically it produces subsets $S(\alpha, \beta)$ which
contains  $NK_\infty(f)$ Therefore in practice we must
repeat the algorithm several times and select only the subset
$S(\alpha,\beta)$
 which contains the same fixed subset all times. The final answer should then be the intersection
 $S :=\bigcap_{\alpha, \beta} S(\alpha,\beta).$

 At step (5) (and (4) in the isolated singularity case, respectively) we
compute Gr\"obner bases in polynomial rings of at most $n+2$
variables.

It is possible to construct also a version of this
algorithm with a certification, however in that case  we have to
compute Gr\"obner bases in polynomial rings of  $2n+1$
variables.
\end{remark}

\begin{remark}
  A similar algorithm can be constructed for
the iterated polar curves method that we use in the first part of
our paper; more steps will be needed. We leave the details to the
reader.
\end{remark}

%%%%%%%%%%%%%%%%%%%%%%%%%%%%%%%%%%%%%%%%%%%%
%%%%%%%%%%%%%%%%%%%%%%%%%%%%%%%%%

%%%%%%%%%%%%%%%%%%%%%%%%%%%
%%%%%%%%%%%%%%%%%%%%%%%%%%%
%%%%%%%%%%%%%%%%%%%%%%%%%%%%%%%%%%%%%%%%

%%%%%%%%%%%%%%%%%
%%%%%%%%%%
%%%%%%%%%%
%%%%%%%%%%%%%%%%%%%%%%%%
%%%%%%%%%%%%%%%%%%%
%%%%%%%%%%%%%%%%%%%%%%%%

\end{document}